\documentclass[11pt,a4paper,twoside,notitlepage]{article}
\usepackage[utf8]{inputenc}
\usepackage[english]{babel}
\usepackage{amsmath}
\usepackage{amsfonts}
\usepackage{amssymb}
\usepackage{makeidx}
\usepackage{graphicx}
\usepackage{lmodern}
\usepackage{kpfonts}
\usepackage[left=2cm,right=2cm,top=2cm,bottom=2cm]{geometry}
\usepackage{tikz}
\usetikzlibrary{matrix,arrows}
\usepackage{amsthm}
\newtheorem{thm}{Theorem}
\newtheorem{cor}[thm]{Corollary}
\newtheorem{lem}[thm]{Lemma}
\newtheorem{prop}[thm]{Proposition}
\theoremstyle{definition}

\theoremstyle{remark}

\newcommand{\G}{\mathcal{G}}
\newcommand{\Dickson}{\mathfrak{D}}
\newcommand{\A}{\mathcal{A}}
\newcommand{\B}{\mathcal{B}}
\newcommand{\Hom}{\textrm{Hom}}

\newcommand{\ann}{\textrm{ann}}
\newcommand{\To}{\rightarrow}
\newcommand{\soc}{\textrm{soc}}
\author{Can Hatipo\u{g}lu\\ \texttt{chatipoglu@alunos.fc.up.pt}}
\title{Stable torsion theories and the injective hulls of simple modules}
\begin{document}
\maketitle
\begin{abstract}
A torsion theoretical characterization of left Noetherian rings $R$ over which injective hulls of simple left modules are locally Artinian is given. Sufficient conditions for a left Noetherian ring to satisfy this finiteness condition are obtained in terms of torsion theories.
\end{abstract}
\section*{Introduction}

A famous open problem in ring theory is Jacobson's conjecture which asks, for a two sided Noetherian ring $R$ with Jacobson radical $J$, whether it is true that $\bigcap_{i=1}^{\infty} J^{i} = 0$. Jategaonkar \cite{jategaonkar} showed in 1974 that fully bounded Noetherian (FBN for short) rings satisfy Jacobson's conjecture. A key step in his proof is that over a FBN ring $R$, finitely generated essential extensions of simple left modules are Artinian. This is equivalent to the property
\begin{center}
 $(\diamond)$ injective hulls of simple left $R$-modules are locally Artinian.
\end{center}
The natural question whether this holds for arbitrary Noetherian rings was  shown not to be true by Musson (see \cite{musson} or \cite{musson82}).

It should be noted that if $R$ is a Noetherian ring which satisfies $(\diamond)$, then $R$ satisfies the Jacobson's conjecture. This makes the property $(\diamond)$ interesting on its own and some Noetherian rings have been tested whether they satisfy this property or not. Recently, interest in this property has increased when Carvalho, Lomp, and Pusat-Y\i lmaz \cite{carvalho-lomp-pusat} considered this property for Noetherian down-up algebras and started the characterization of these algebras with property $(\diamond)$  by giving a partial answer. This characterization has been then completed in the following works of Carvalho and Musson, and Musson \cite{carvalho-musson}, \cite{musson2012}. These results have been followed by a characterization of finite dimensional solvable Lie superalgebras $\mathfrak{g}$ over an algebraically closed field of characteristic zero whose enveloping algebra $U(\mathfrak{g})$ satisfies $(\diamond)$ \cite{hatipoglu-lomp}. Most recently, a complete characterization of Ore extensions $K[x][y;\alpha, d]$ of $K[x]$ with property $(\diamond)$ has been obtained in \cite{carvalho-hatipoglu-lomp}.

There is a way in which torsion theories and property $(\diamond)$ can be linked. We prove in Proposition~\ref{connection} below that for a left Noetherian ring $R$, injective hulls of simple left $R$-modules are locally Artinian if and only if Dickson's torsion theory is stable. This connection makes it possible to carry the study of Noetherian rings satisfying property $(\diamond)$ to the area of stable torsion theories. Using the results from stable torsion theories, we are able to obtain more examples of rings which satisfy property $(\diamond)$.

The organization of the paper is as follows. In the first section we briefly recall the general notions of torsion theories. We then consider in the next section stable torsion theories and obtain certain conditions on rings which lead to property $(\diamond)$. The next part of the paper is devoted to some classes of rings that satisfy these certain conditions. 

\section{Generalities on torsion theories}
Let $R$ be an arbitrary associative ring with unity and let $R$-mod denote the category of left $R$-modules. Dickson had the idea of carrying the notion of torsion in abelian groups to abelian categories, and he defined in \cite{dickson} a \textit{torsion theory} $\tau$ on $R$-mod to be a pair $\tau = (T_{\tau},F_{\tau})$ of classes of left $R$-modules, satisfying the following properties:
\begin{enumerate}
	\item[(i)] $T_{\tau} \cap F_{\tau} = \{0\}$,
	\item[(ii)] $T_{\tau}$ is closed under homomorphic images,
	\item[(iii)] $F_{\tau}$ is closed under submodules,
	\item[(iv)] For each $M$ in $R$-mod, there exist $F \in F_{\tau}$ and $T \in T_{\tau}$ such that $M/T \cong F$.
\end{enumerate}
We call $T_{\tau}$ the class of $\tau$-\textit{torsion} modules and $F_{\tau}$ the class of $\tau$-\textit{torsionfree} modules.

Let $\mathcal{A}$ and $\mathcal{B}$ be nonempty classes of left $R$-modules. If $\A = \{ M \mid \Hom_{R}(M,N) = 0 \textrm{ for all } N \in \B\}$ then $\A$ is said to be the \textit{left orthogonal complement} of $\B$. Similarly, if $\B = \{ N \mid \Hom_{R}(M,N) = 0 \textrm{ for all } M \in \A\}$ then $\B$ is said to be the \textit{right orthogonal complement} of $\A$. We say that the pair $(\A, \B)$ is a \textit{complementary pair} whenever $\A$ is the left orthogonal complement of $\B$ and $\B$ is the right orthogonal complement of $\A$. In particular, if $\tau$ is a torsion theory then $(T_{\tau}, F_{\tau})$ is a complementary pair and every such pair defines a torsion theory.

An immediate consequence of the definition is that the class of torsion modules for a torsion theory $\tau$ is closed under extension. For, if
$$0 \To A \To M \To B \To 0$$
is a short exact sequence of left $R$-modules with $A, B$ are $\tau$-torsion, then for any $\tau$-torsionfree left $R$-module $F$, the corresponding short exact sequence
$$0 = \Hom(B, F) \To \Hom(M,F) \To \Hom(A,F) = 0$$
implies that $M$ is also $\tau$-torsion. Similarly, the class of torsionfree modules is closed under extensions too. While it is not required in the definition of a torsion theory, we will be working with torsion theories such that the class of torsion modules is closed under submodules. Such torsion theories are called \textit{hereditary}.

A nonempty set $\mathcal{L}$ of left ideals of a ring $R$ which satisfies the following conditions is called a \textit{Gabriel filter}:
\begin{itemize}
	\item[(i)] $I \in \mathcal{L}$ and $a \in R$ implies $\ann_{R}(a + I)$ belongs to $\mathcal{L}$.
	\item[(ii)] If $I$ is a left ideal and $J \in \mathcal{L}$ is such that $\ann_{R}(a + I) \in \mathcal{L}$ for all $a \in J$, then $I \in \mathcal{L}$. 
\end{itemize}
Hereditary torsion theories in $R$-mod and Gabriel filters in $R$ are in one-to-one correspondence \cite[Theorem VI.5.1]{stenstrom}.
\subsection{Goldie's torsion theory}

Let $M$ be a left $R$-module. An element $m \in M$ is called a \textit{singular element} of $M$ if $\ann_{R}(m) = \{r \in R \mid rm = 0\}$ is an essential left ideal of $R$. The collection $Z(M)$ of all singular elements of $M$ is a submodule of $M$ called the \textit{singular submodule} of $M$. A module $M$ is called \textit{singular} if $Z(M) = M$ and it is called \textit{nonsingular} if $Z(M) = 0$. The class $F_{\G}$ of all nonsingular left $R$-modules forms a torsionfree class for a hereditary torsion theory on mod-$R$. We call this \textit{Goldie's torsion theory} and denote it by $\G$ \cite{goldie, teply69}.

For any right $R$-module $M$, its Goldie torsion submodule is $t_{\G}(M) = \{m \in M \mid m + Z(M) \in Z(M/Z(M))\}$. The Gabriel filter corresponding to Goldie's torsion theory is the set of all essential left ideals $L$ of $R$ such that there exists an essential left ideal $L'$ of $R$ such that for every $x \in L'$, $\ann_{R}(x + L) = \{r \in R \mid rx \in L\}$ is essential in $R$. Hence, Goldie's torsion class $T_{\G}$ is precisely the class of modules with essential singular submodule, and corresponding torsionfree class $F_{\G}$ is the class of nonsingular modules.

\subsection{Generation \& cogeneration of torsion theories, Dickson's torsion theory}

Let $\mathcal{C}$ be a class of left $R$-modules. If $\mathcal{F}$ is a right orthogonal complement of $\mathcal{C}$ and $\mathcal{T}$ is a left orthogonal complement of $\mathcal{F}$, then the pair $(\mathcal{T}, \mathcal{F})$ is a torsion theory in $R$-mod, called the \textit{torsion theory generated by} $\mathcal{C}$. If $\mathcal{T}$ is the left orthogonal complement of $\mathcal{C}$ and $\mathcal{F}$ is the right orthogonal complement of $\mathcal{T}$, then the pair $(\mathcal{T}, \mathcal{F})$ is a torsion theory, called the \textit{torsion theory cogenerated by} $\mathcal{C}$.

Let $\mathcal{S}$ be a representative class of nonisomorphic simple left $R$-modules. Then the torsion theory $\Dickson$ generated by $\mathcal{S}$ is called \textit{Dickson's torsion theory}. The class of $\Dickson$-torsionfree left $R$-modules are the right orthogonal complements of simple left $R$-modules while the class of $\Dickson$-torsion left $R$-modules are the left orthogonal complements of the class of $\Dickson$-torsion modules. In particular, every simple left $R$-module is $\Dickson$-torsion. Hence the class of all $\Dickson$-torsionfree modules consists of all soclefree left $R$-modules. Moreover, if $M \in T_{\Dickson}$ then $M$ is an essential extension of its socle.

A left $R$-module $M$ is called \textit{semi-Artinian} if for every submodule $N \neq M$, $M/N$ has nonzero socle. We first show that the class of $\Dickson$-torsion left $R$-modules is exactly the class of semi-Artinian left $R$-modules.
\begin{lem}\label{d-torsions_are_semi-artinian}
A left $R$-module $M$ is $\Dickson$-torsion if and only if it is semi-Artinian.
\end{lem}
\begin{proof}
Let $M$ be a semi-Artinian left $R$-module and $F$ be a $\Dickson$-torsionfree left $R$-module. Then $\soc(F) = 0$ by definition. We show that $\Hom(M,F) = 0$. Every nonzero $R$-homomorphism $f \in \Hom(M,F)$ gives rise to an injective $R$-homomorphism $f' : M/\ker f \To F$. Let $S$ be the socle of $M/\ker f$. Then $f'(S) \subseteq \soc(F) = 0$, hence $S=0$. But this implies that $M = \ker f$. Hence $f = 0$.

Now suppose that $M$ is a $\Dickson$-torsion left $R$-module. Since $\Hom_{R}(M,M/N) \neq 0$ for every proper submodule $N$ of $M$, $M/N$ cannot be $\Dickson$-torsionfree. Hence $M/N$ has nonzero socle for every submodule $N$  of $M$ and so $M$ is semi-Artinian.
\end{proof}
Recall that a module has finite length if and only if it is both Noetherian and Artinian. In fact, we can still have finite length if the module is Noetherian and semi-Artinian. 
\begin{lem}\label{finite-length}\cite[Proposition VIII.2.1]{stenstrom}
A left $R$-module $M$ has finite length if and only if it is Noetherian and semi-Artinian.
\end{lem}
%
\section{Stable torsion theories}

A hereditary torsion theory $\tau$ on $R$-mod is called \textit{stable} if its torsion class is closed under injective hulls. One of the equivalent conditions for $\Dickson$ to be stable is that modules with essential socle are $\Dickson$-torsion \cite[4.13]{dickson}.

Dickson characterized those rings for which Dickson's torsion theory in the category $R$-Mod is stable. Indeed he considered property $(\diamond)$ for any abelian category with injective envelopes. Translating his results to the language of the present paper, for a left Noetherian ring we obtain a connection between the stability of Dickson's torsion theory and property $(\diamond)$ in the following, which is the main result of the paper:
\begin{prop}\label{connection}
The following are equivalent for a left Noetherian ring $R$.
\begin{itemize}
	\item[(i)] $R$ satisfies property $(\diamond)$;
	\item[(ii)] Dickson's torsion theory is stable;
	\item[(iii)] Any $\Dickson$-torsion $R$-module can be embedded in a $\Dickson$-torsion injective $R$-module;
	\item[(iv)] Any injective $R$-module $A$ decomposes as $A = A_{t} \oplus F$, where $A_{t}$ is the $\Dickson$-torsion part of $A$ and $F$ is unique up to isomorphism and has no socle;
	\item[(v)] If $A$ is an essential extension of its socle, then its $\Dickson$-torsion;
	\item[(vi)] For any left $R$-module $A$, its torsion part $A_{t}$ is the unique maximal essential extension in $A$ of its socle $soc(A)$.
\end{itemize}
\end{prop}
\begin{proof}
$(i) \Rightarrow (ii)$ We show that $(\diamond)$ implies the stability of $\Dickson$. Let $M$ be a $\Dickson$-torsion left $R$-module. Then $M$ has an essential socle and so its injective hull $E(M)$ is a direct sum of injective hulls of simple left $R$-modules because $R$ is Noetherian. Then $E(M)$ is locally Artinian by assumption. We show that $E(M)$ is $\Dickson$-torsion. Let $f: E(M) \To F$ be an $R$-module homomorphism, where $F$ is a left $R$-module with zero socle. For any $x \in E(M)$, since $Rx$ is Artinian, the restriction $f: Rx \To F$ is zero. It follows that $f$ is zero and thus $E(M)$ is also $\Dickson$-torsion by definition.

$(v) \Rightarrow (i)$ Let $S$ be a simple left $R$-module and $E(S)$ be its injective hull. Let $0 \neq F \leq E(S)$ be a finitely generated submodule of $E(S)$. Then $S\leq_{e}F\leq_{e}E(S)$ and having an essential simple submodule, $F$ has an essential socle. This means, by assumption, that $F$ is $\Dickson$-torsion, \textit{i.e.} $F$ is semi-Artinian by Lemma~\ref{d-torsions_are_semi-artinian}. Since it is a finitely generated module over a left Noetherian ring, $F$ is Noetherian as well. By Lemma~\ref{finite-length} $F$ is Artinian. Thus $E(S)$ is locally Artinian.

The equivalence of (ii - iv) follows from \cite[4.13]{dickson}.
\end{proof}
%

Hence, over a Noetherian ring, the stability of Dickson's torsion theory is a necessary and sufficient condition for property $(\diamond)$. We will be looking for cases in which Dickson's torsion theory is stable for a left Noetherian ring $R$. There are two such cases which imply the stability of Dickson's torsion theory, but first we should introduce a partial order among the torsion theories defined on $R$-mod.

For a ring $R$ we denote the family of all hereditary torsion theories defined on $R$-mod by $R$-tors. Note that $R$-tors corresponds bijectively to a set, see for example \cite[Proposition 4.6]{golan}. We define a partial order in $R$-tors with the help of the following result:

\begin{prop}\cite[Proposition 2.1]{golan}
For torsion theories $\tau$ and $\sigma$ on $R$-mod the following conditions are equivalent:
\begin{itemize}
\item[(a)] Every $\tau$-torsion left $R$-module is $\sigma$-torsion;
\item[(b)] Every $\sigma$-torsionfree left $R$-module is $\tau$-torsionfree.
\end{itemize}
\end{prop}

In case $\tau$ and $\sigma$ are torsion theories on $R$-mod which satisfy the equivalent conditions of the above proposition, we say that $\tau$ is a \textit{specialization} of $\sigma$ and that $\sigma$ is a \textit{generalization} of $\tau$. We denote this situation by $\tau \leq \sigma$. This defines a partial order in $R$-tors. For example, with respect to this ordering, Goldie's torsion theory is the smallest torsion theory in which every cyclic singular left $R$-module is torsion and Dickson's torsion theory is the smallest torsion theory in which every simple left $R$-module is torsion.

\section{Cyclic singular modules with nonzero socle}
We now give a sufficient condition for a torsion theory to be stable. The following result is present in the proof of Proposition 1 in \cite{teply71} but it is not given explicitly. We record it as a lemma and give its proof for the convenience of the reader.

\begin{lem}\label{generaliztaion_of_goldie_is_stable}
Any generalization of Goldie's torsion theory $\G$ is stable.
\end{lem}

\begin{proof}
Suppose that $(T,F)$ is a torsion theory which is a generalization of $\G$. For all $M \in T$, since $M$ is essential in its injective hull $E(M)$, $E(M)/M$ is Goldie torsion. Since $T_{\G} \subseteq T$, $E(M)/M$ also belongs to $T$. Since $T$ is closed under extensions, it follows that $E(M)$ also belongs to $T$ and hence $(T,F)$ is stable.
\end{proof}

In particular, Dickson's torsion theory is stable if it is a generalization of Goldie's torsion theory. This can be summarized as follows:
\begin{cor}
If every cyclic singular left $R$-module has a nonzero socle then Dickson's torsion theory is stable.
\end{cor}
\begin{proof}
By assumption, every cyclic singular left $R$-module has the property that every epimorphic image has nonzero socle. Thus every such module belongs to Dickson's torsion class. Since Goldie's torsion theory is the smallest torsion theory in which every cyclic singular module is torsion, it follows that Dickson's torsion theory is a generalization of Goldie's torsion theory, hence it is stable.
\end{proof}

The rings $R$ such that every cyclic singular left $R$-module has a nonzero socle are called \textit{C-rings} in \cite{renault}. Alternatively, they are characterized as the rings over which neat submodules are closed in $R$-mod, where a submodule $N$ of a left $R$-module $M$ is called \textit{neat} if any simple module $S$ is projective relative to the projection $M \To M/N$ and a submodule $N$ of a module $M$ is called \textit{closed} in $M$ if it is not essential in any submodule of $M$. 

In his 1981 paper \cite{smith}, P.F. Smith considers collections of left ideals to test injectivity. For a nonempty collection $\mathcal{C}$ of left ideals of a ring $R$, we say that a left $R$-module $M$ is $\mathcal{C}$-\textit{injective} if for every left ideal $I$ from $\mathcal{C}$, every $R$-homomorphism $I \rightarrow M$ can be lifted to an $R$-homomorphism $R \rightarrow M$. Combining these results, we get a list of different characterizations, given in the following result:

\begin{prop}\cite[10.10]{lifting_modules}\cite[Lemma 4]{smith}
Let $R$ be a ring and $\mathcal{M}ax$ be the collection of maximal left ideals of $R$. The following conditions are equivalent.
\begin{itemize}
	\item[(a)] $R$ is a left C-ring;
	\item[(b)] Every singular module is semi-Artinian;
	\item[(c)] Every neat left ideal of $R$ is closed;
	\item[(d)] A left ideal of $R$ is closed if and only if it is neat;
	\item[(e)] For every left $R$-module, closed submodules are neat;
	\item[(f)] Every $\mathcal{M}ax$-injective left $R$-module is injective.
\end{itemize}
\end{prop}

%
%
Hence, a Noetherian ring satisfying any of the equivalent conditions of the above proposition satisfies property $(\diamond)$.
\section{Torsionfree projective modules}
Teply considers in \cite{teply71} for an arbitrary torsion theory $\tau$ the following property:
\begin{center}
(P) every nonzero $\tau$-torsionfree module contains a nonzero projective submodule.
\end{center}
It turns out that this property is closely related to Goldie's torsion theory in the sense that if a torsion theory $\tau$ satisfies condition (P) then $\tau$ is a generalization of Goldie's torsion theory by \cite[Proposition 1]{teply71}. This in turn means that any torsion theory satisfying (P) is stable by Lemma~\ref{generaliztaion_of_goldie_is_stable}.

In particular, when this is applied to Dickson's torsion theory we find out that Dickson's torsion theory is stable when every nonzero soclefree module contains a nonzero projective submodule. This means that every ring $R$ such that Dickson's torsion theory satisfies condition (P) in $R$-Mod is a C-ring. Teply proves the following result as an equivalent condition to (P): 

\begin{prop}\cite[Proposition 2]{teply71}
Let $\tau$ be a torsion theory in $R$-Mod. $\tau$ satisfies condition (P) if and only if every nonzero torsionfree left ideal contains a nonzero projective left ideal and $\tau$ is a generalization of Goldie's torsion theory.
\end{prop}

Then, since over a C-ring Dickson's torsion theory is a generalization of Goldie's torsion theory, for a $C$-ring $R$, Dickson's torsion theory satisfies (P) if and only if every soclefree left ideal contains a nonzero projective left ideal. As a concrete example we can consider Rickart rings, where a ring $R$ is called a \textit{left Rickart ring} if and only if every principal left ideal of $R$ is projective as a left $R$-module \cite[\S 7D]{lam}.

\section{Remarks}

We finish with a list of remarks.

(1) While it is true that every Noetherian C-ring has property $(\diamond)$, there are Noetherian rings which satisfy $(\diamond)$ but are not C-rings. For example, the ring of polynomials $R = K[x,y]$ in two indeterminates over a field $K$ is a commutative Noetherian domain and hence satisfies $(\diamond)$. The ideal $I = \langle x \rangle$ is essential in $R$ since $R$ is a domain, but the singular module $M = R/I$ has zero socle and hence $R$ is not a C-ring.

(2) Let $R$ be a Noetherian C-domain. Let $0 \neq P$ be a prime ideal in $R$. Since $R$ is a domain, $P$ is essential in $R$. Hence $R/P$ has finite length. It follows that $R/P$ is simple and $P$ is a maximal ideal, so every Noetherian C-domain satisfies the property that each prime ideal is maximal.

(3) As mentioned in the introduction, a complete characterization of Ore extensions $S = K[x][y;\sigma, d]$ with property $(\diamond)$ has been obtained in \cite{carvalho-hatipoglu-lomp}. According to their result, such an Ore extension has property $(\diamond)$ if and only if $\sigma \neq 0$ has finite order or $\sigma = 1$ and $d$ is locally nilpotent. In this case, $S$ is isomorphic to either the quantum plane or the first Weyl algebra with $q$ is a root of unity. While the first Weyl algebra is a hereditary ring, the quantum plane is not a C-ring since the prime ideal $\langle x \rangle$ is not maximal.

(4) Another class of rings over which the cyclic singular modules have nonzero socle is the so called class of SI-rings. A ring $R$ is said to be a \textit{left SI-ring} if every singular left $R$-module is injective. SI-rings satisfy the stronger property that $R/E$ is semisimple for every essential left ideal $E$ of $R$. Of course each SI-ring is a C-ring and a left SI-ring is hereditary. For a list of equivalent conditions for a ring $R$ to be a left SI-ring, see \cite[17.4]{extending_modules}.

(5) Every left hereditary Noetherian ring is a left $C$-ring \cite[5.4.5]{mcconnell-robson}, but the converse is not true in general, because there are nonhereditary $C$-rings. For example, a commutative Noetherian domain $R$ which is not integrally closed such that every nonzero prime ideal is maximal is a $C$-ring which is not hereditary. To see this, first note that $R$ is not a Dedekind domain because it is not integrally closed. On the other hand, Dedekind domains are alternatively characterized as integral domains which are hereditary, hence $R$ cannot be hereditary. 

(6) A particular case in which the converse holds is when $R$ is an SC-ring, \textit{i.e.}, when singular modules are continuous. In \cite[17.4]{extending_modules} it is proved, among other things, that $R$ is a left SI-ring if and only if $R$ is a left hereditary left SC-ring. Another case in which this holds is provided by Faith in \cite{faith}. A ring $R$ is called a left \textit{QI-ring} if every quasi injective left $R$-module is injective. Then, Faith proves that a left QI left C-ring is hereditary \cite[Theorem 18]{faith}.

\section*{Acknowledgment}

Parts of this paper have appeared in the author's doctoral dissertation at the University of Porto. The author wishes to thank his supervisor Professor Christian Lomp for his inspiration and guidance. This research was supported by \textit{Funda\c{c}\~{a}o para a Ci\^{e}ncia e a Tecnologia - FCT} by the grant SFRH/BD/33696/2009.

\bigskip

\flushleft Can Hatipo\u{g}lu\\
Center of Mathematics\\
University of Porto\\
Rua do Campo Alegre 687\\
4169-007, Porto, Portugal


\begin{thebibliography}{99}

\bibitem{carvalho-hatipoglu-lomp} Carvalho, Paula A. A. B.; Hatipoglu, Can; Lomp, Christian. \textit{Injective hulls of simple modules over differential operator rings.} arXiv:1211.2592.

\bibitem{carvalho-lomp-pusat} Carvalho, Paula A. A. B.; Lomp, Christian; Pusat-Yilmaz, Dilek. \textit{Injective modules over down-up algebras}. Glasg. Math. J. 52 (2010), no. A, 53-59. 
\bibitem{carvalho-musson} Carvalho, Paula A. A. B.; Musson, Ian M. \textit{Monolithic modules over Noetherian rings}. Glasg. Math. J. 53 (2011), no. 3, 683-692.

\bibitem{lifting_modules} Clark, John; Lomp, Christian; Vanaja, Narayanaswami; Wisbauer, Robert. \textit{Lifting modules}. Supplements and projectivity in module theory. Frontiers in Mathematics. Birkh\"{a}user Verlag, Basel, 2006. xiv+394 pp. ISBN: 978-3-7643-7572-0; 3-7643-7572-8.

\bibitem{dickson} Dickson, Spencer E. \textit{A torsion theory for Abelian categories}. Trans. Amer. Math. Soc. 121 1966 223-235.

\bibitem{extending_modules}  Nguyen Viet Dung; Dinh Van Huynh; Smith, Patrick F.; Wisbauer, Robert. \textit{Extending modules}. With the collaboration of John Clark and N. Vanaja. Pitman Research Notes in Mathematics Series, 313. Longman Scientific \& Technical, Harlow; copublished in the United States with John Wiley \& Sons, Inc., New York, 1994. xviii+224 pp. ISBN: 0-582-25382-9. 


\bibitem{faith} Faith, Carl. \textit{On hereditary rings and Boyle's conjecture}. Arch. Math. (Basel) 27 (1976), no. 2, 113-119. 

\bibitem{golan} Golan, Jonathan S. \textit{Torsion theories}. Pitman Monographs and Surveys in Pure and Applied Mathematics, 29. Longman Scientific \& Technical, Harlow; John Wiley $\&$ Sons, Inc., New York, 1986. xviii+651 pp. ISBN: 0-582-99808-5. 

\bibitem{goldie} Goldie, A. W. \textit{Torsion-free modules and rings}. 
J. Algebra 1 1964 268-287. 

\bibitem{hatipoglu-lomp} Hatipo\u{g}lu, Can; Lomp, Christian. \textit{Injective hulls of simple modules over finite dimensional nilpotent complex Lie superalgebras}. J. Algebra 361 (2012), 79-91. 

\bibitem{jategaonkar} Jategaonkar, Arun Vinayak. \textit{Jacobson's conjecture and modules over fully bounded Noetherian rings}. J. Algebra 30 (1974), 103-121. 


\bibitem{mcconnell-robson} McConnell, J. C.; Robson, J. C. \textit{Noncommutative Noetherian rings}. With the cooperation of L. W. Small. Revised edition. Graduate Studies in Mathematics, 30. American Mathematical Society, Providence, RI, 2001. xx+636 pp. ISBN: 0-8218-2169-5 16P40 (16-02).

\bibitem{lam} Lam, T. Y. \textit{Lectures on modules and rings}. (English summary) Graduate Texts in Mathematics, 189. Springer-Verlag, New York, 1999. xxiv+557 pp. ISBN: 0-387-98428-3. 

\bibitem{musson}  Musson, I. M. \textit{Injective modules for group rings of polycyclic groups}. I, II. Quart. J. Math. Oxford Ser. (2) 31 (1980), no. 124, 429-448, 449-466. 

\bibitem{musson82} Musson, I. M. \textit{Some examples of modules over Noetherian rings}. Glasgow Math. J. 23 (1982), no. 1, 9-13. 

\bibitem{musson2012} Musson, Ian M. \textit{Finitely generated, non-artinian monolithic modules}. New trends in noncommutative algebra, 211-220, Contemp. Math., 562, Amer. Math. Soc., Providence, RI, 2012. 

\bibitem{renault} Renault, Guy. \textit{Etude des sous-modules compl\'{e}ments dans un module}. Bull. Soc. Math. France M\'{e}m. 9 1967 79 pp.



\bibitem{smith} Smith, P. F. \textit{Injective modules and prime ideals}. Comm. Algebra 9 (1981), no. 9, 989-999.

\bibitem{stenstrom} Stenstr\"{o}m, Bo. \textit{Rings of quotients}. 
Die Grundlehren der Mathematischen Wissenschaften, Band 217. An introduction to methods of ring theory. Springer-Verlag, New York-Heidelberg, 1975. viii+309 pp. 

\bibitem{teply69} Teply, Mark L. \textit{Some aspects of Goldie's torsion theory}. Pacific J. Math. 29 1969 447-459. 

\bibitem{teply71}Teply, Mark L. \textit{Torsionfree projective modules}. Proc. Amer. Math. Soc. 27 1971 29-34.
\end{thebibliography}
\end{document}